\documentclass[preprint]{elsarticle}
\usepackage{array}
\usepackage{amsmath, amsthm, amssymb ,amsfonts}					% Standard packages
\usepackage{multirow}
\usepackage{graphicx}
\usepackage{hyperref}
\usepackage{float}
\usepackage{enumitem}
\usepackage{kbordermatrix}									% Fancy matrices
\usepackage[table]{xcolor}									% Colored Tables
\usepackage{epsfig}										% ELA

\allowdisplaybreaks

\newtheorem{thm}{Theorem}[section]
\newtheorem{lem}[thm]{Lemma}
\newtheorem{cor}[thm]{Corollary}

\theoremstyle{definition}										% Style for subsequent thm environments

\newtheorem{rem}[thm]{Remark}
\newtheorem{ex}[thm]{Example}

\numberwithin{equation}{section}								% Number eqns within section

\newcommand{\bb}[1]{\mathbb{#1}}								% Blackboard boldface for R, C, N, Q, etc.
\newcommand{\ii}{\textup{i}}									% \textup{i} := \sqrt{-1}
\newcommand{\set}[4]{\left\{ #1 \right\}_{#2 = #3}^{#4}}				% Indexed set

\newcommand{\diag}[1]{\operatorname{\rm diag}\left( #1 \right)}			% diagonal matrix
\newcommand{\circulant}[1]{\operatorname{\rm circ} #1}					% Circulant
\newcommand{\sr}[1]{\rho\left(#1\right)}							% spectral radius
\newcommand{\sig}[1]{\sigma \left( #1 \right)}						% multi-set of eigenvalues
\newcommand{\peri}[1]{\pi\left(#1\right)}							% Peripheral spectrum
\newcommand{\dg}[1]{\Gamma \left( #1 \right)}						% directed graph (gamma}

\newcommand{\mat}[2]{M_{#1}(#2)}								% nxn matrices over field
\newcommand{\jordan}[2]{J_{#1}{\left( #2 \right)}}					% Jordan block

\newcommand{\inv}[1]{#1^{-1}}								% inverse  
\newcommand{\hyp}[2]{\hyperref[#1]{{\rm #2 \ref*{#1}}}}				% Hyperlink
\journal{Linear Algebra and its Applications}

\begin{document}
% Front matter---------------------------------------------------------------------------------------------------------------------------------------------------------------------
\begin{frontmatter}
\title{Jordan Chains of \emph{h}-cyclic Matrices} %\tnoteref{title}}
%\tnotetext[label1]{Blah}

\author[addy1]{Judith J. McDonald}										
\ead{jmcdonald@math.wsu.edu}
\ead[url]{http://www.math.wsu.edu/math/faculty/jmcdonald/}

\author[addy2]{Pietro Paparella\corref{corpp}}
\ead{ppaparella@wm.edu}
\ead[url]{http://ppaparella.people.wm.edu/}

\cortext[corpp]{Corresponding author.}

\address[addy1]{Department of Mathematics, Washington State University, Pullman, WA 99164-1113, U.S.A.}
\address[addy2]{Department of Mathematics, College of William \& Mary, Williamsburg, VA 23187-8795, U.S.A.}

\begin{abstract}
Arising from the classification of the matrix-roots of a nonnegative imprimitive irreducible matrix, we present results concerning the Jordan chains of an \textit{h}-cyclic matrix. We also present ancillary results applicable to nonnegative imprimitive irreducible matrices and demonstrate these results via examples.
\end{abstract}

\end{frontmatter}

 %-------------------------------------------------------------------------------------------------------------------------------------------------------------------------------------
\section{Introduction}
%-------------------------------------------------------------------------------------------------------------------------------------------------------------------------------------

The study of nonnegative matrices has its roots in the Perron-Frobenius Theorem, which asserts that a nonnegative irreducible matrix has a positive eigenvector associated with its spectral radius.  This study has been extended to include results on reducible nonnegative matrices. For surveys of some of these results see \cite{bp1994}, \cite{h1999}, and \cite{s1986}. Spectral properties of nonnegative matrices have proven very useful in the study of other related classes of matrices, such as M-matrices (see \cite{bp1994}) and eventually nonnegative matrices (see, e.g., \cite{cnm2002, cnm2004, {f1978}, {h2009}, {jt2004}, {m2003}, {n2006}, {nt2009}, {mpt2014}, {zt1999},{zm2003}}).

The focus of this article is to exploit the cyclicity of the spectrum of $h$-cyclic matrices to glean more information about the generalized eigenvectors of nonnegative and eventually nonnegative matrices. In \cite{t1999}, Tam observed that, for nonnegative irreducible imprimitive matrices, it is possible to predict the structure of the peripheral eigenvectors from the Perron vector.  This idea is now extended to the generalized eigenvectors and Jordan chains of all the eigenvalues associated with any $h$-cyclic matrix, and irreducible nonnegative matrices in particular. We use these results in our paper on matrix-roots of imprimitive irreducible nonnegative matrices \cite{mp2014b}.

%-------------------------------------------------------------------------------------------------------------------------------------------------------------------------------------
\section{Notation and Definitions}
%-------------------------------------------------------------------------------------------------------------------------------------------------------------------------------------

Denote by $\ii$ the imaginary unit, i.e., $\ii := \sqrt{-1}$. When convenient, an indexed set of the form $\{ x_i, x_{i+1}, \dots, x_{i+j} \}$ is abbreviated to $\left\{ x_k \right\}_{k=i}^{i+j}$. 

For $h \in \bb{N}$, $h>1$,  
\begin{align}
R(h) &:= \{ 0, 1, \dots, h - 1 \}													\nonumber 			\\
\omega &:= \exp{\left( 2 \pi \ii/h \right)} \in \bb{C}, 										\nonumber 			\\
\Omega_h &:= \set{\omega^k}{k}{0}{h-1} \subseteq \bb{C}, 								\label{bigomegah}  % 2.1  
\end{align}
and
\begin{align}														
\nu_h := 
\left( 1, \omega, \dots, \omega^{h-1} \right) \in \bb{C}^n.									\label{omegastar} % 2.2
\end{align}

Denote by $\mat{n}{\bb{C}}$ ($\mat{n}{\bb{R}}$) the algebra of complex (respectively, real) $n \times n$ matrices. Given $A \in \mat{n}{\bb{C}}$, the \textit{spectrum} of $A$ is denoted by $\sig{A}$; the \emph{spectral radius} of $A$ is denoted by $\rho = \sr{A}$; and the \emph{peripheral spectrum}, denoted by $\peri{A}$, is the multi-set given by 
\begin{align*}
\peri{A} = \{ \lambda \in \sig{A} : |\lambda| = \rho \}.
\end{align*} 
The (block) $(i,j)$-entry of $A$ is denoted by $a_{ij}$ or $[A]_{ij}$ and the (block) entries of $A$ are denoted by $[ a_{ij}]$ or $[a_{ij}]_n^{i,j=1}$.

 The \emph{direct sum} of the matrices $A_1, \dots, A_k$, where $A_i \in \mat{n_i}{\bb{C}}$, denoted by $A_1 \oplus \dots \oplus A_k$, $\bigoplus_{i=1}^k A_i$, or $\diag{A_1,\dots,A_k}$, is the $n \times n$ matrix 
\[ 
 \left[
 \begin{array}{ccc}
 A_1 &  & \multirow{2}{*}{\Large 0} \\
 \multirow{2}{*}{\Large 0} & \ddots &  \\
  &  & A_k
 \end{array}
 \right], 
\]
where $n = \sum_{i=1}^k n_i$.

For $\lambda \in \bb{C}$, $\jordan{n}{\lambda}$ denotes the $n \times n$ \emph{Jordan block} with eigenvalue $\lambda$. For $A \in \mat{n}{\bb{C}}$, denote by $J = \inv{Z} A Z = \bigoplus_{i=1}^t \jordan{n_i}{\lambda_i} = \bigoplus_{i=1}^t J_{n_i}$, where $\sum n_i = n$, a Jordan canonical form of $A$. 

For $c = (c_1, \dots, c_n ) \in \bb{C}^n$, the \emph{circulant matrix} or \emph{circulant} of $c$, denoted by $\circulant{(c)}$, is the $n \times n$ matrix 
\begin{align*}
\begin{bmatrix}
c_1 & c_2 & \cdots & c_n 			\\
c_n & c_1 & \cdots & c_{n-1}		\\
\vdots & \vdots & \ddots & \vdots 	\\
c_2 & c_3 & \cdots & c_1 	
\end{bmatrix}.
\end{align*}
For $n \in \bb{N}$, let 
\[ K_n := \circulant{(\overbrace{0,1,0,\dots,0}^n)}
=\begin{bmatrix}
 & 1 & & \\
 & & \ddots \\
 & & & 1 \\
1 
\end{bmatrix} \in \mat{n}{\bb{R}}. \]

For $A$, $B \in \mat{n}{\bb{C}}$, the \emph{hadamard product} of $A$ and $B$, denoted by $A \circ B$, is the $n \times n$ matrix whose $(i,j)$-entry is $a_{ij} b_{ij}$. 

\section{Preliminaries}

For notation and definitions concerning the \emph{combinatorial stucture of a matrix}, i.e., the location of the zero-nonzero entries of a matrix, we follow \cite{br1991} and \cite{h2009}; for further results concerning \emph{combinatorial matrix theory}, see \cite{br1991} and references therein.

A \emph{directed graph} (or simply \emph{digraph}) $\Gamma = (V,E)$ consists of a finite, nonempty set $V$ of \emph{vertices}, together with a set $E \subseteq V \times V$ of \emph{arcs}. For $A \in \mat{n}{\bb{C}}$, the \emph{directed graph} (or simply \emph{digraph}) of $A$, denoted by $\Gamma = \dg{A}$, has vertex set $V = \{ 1, \dots, n \}$ and arc set $E = \{ (i, j) \in V \times V : a_{ij} \neq 0\}$. If $R$, $C \subseteq \{1,\dots, n\}$, then $A[R|C]$ denotes the submatrix of $A$ whose rows and columns are indexed by $R$ and $C$, respectively. %If $R = C$, then $A[R|R]$ is abbreviated to $A[R]$. For a digraph $\Gamma = (V,E)$ and $W \subseteq V$, the induced \emph{subdigraph} $\Gamma[W]$ is the digraph with vertex set $W$ and arc set $\{(u,v) \in E : u,v \in W\}$. 

A digraph $\Gamma$ is \emph{strongly connected} if for any two distinct vertices $u$ and $v$ of $\Gamma$, there is a walk in $\Gamma$ from $u$ to $v$ (following \cite{br1991}, we consider every vertex of $V$ as strongly connected to itself). For a strongly connected digraph $\Gamma$, the \emph{index of imprimitivity} is the greatest common divisor of the lengths of the closed walks in $\Gamma$. A strong digraph is \emph{primitive} if its index of imprimitivity is one, otherwise it is \emph{imprimitive}. 

For $n \geq 2$, a matrix $A \in \mat{n}{\bb{C}}$,  is \emph{reducible} if there exists a permutation matrix $P$ such that
\begin{align*}
P^\top A P =
\begin{bmatrix}
A_{11} & A_{12} \\
0 & A_{22}
\end{bmatrix},
\end{align*}
where $A_{11}$ and $A_{22}$ are nonempty square matrices and $0$ is a zero block. If $A$ is not reducible, then A is called \emph{irreducible}. The connection between irreducibility and the digraph of $A$ is as follows: $A$ is irreducible if and only if $\dg{A}$ is strongly connected\footnote{Following \cite{br1991}, vertices are strongly connected to themselves so we take this result to hold for all $n \in \bb{N}$ and not just $n\in \bb{N}$, $n \geq 2$. In particular, a $1\times1$ block with entry 0 is considered irreducible in this article.} (see, e.g., \cite[Theorem 3.2.1]{br1991} or \cite[Theorem 6.2.24]{hj1990}). 

For $h \geq 2$, a digraph $\Gamma = ( V, E )$ is \emph{cyclically $h$-partite} if there exists an ordered partition $\Pi = (\pi_1,\dots, \pi_h)$ of $V$ into $h$ nonempty subsets such that for each arc $(i, j) \in E$, there exists $\ell \in \{ 1, \dots, h \}$ such that $i \in \pi_\ell$ and $j \in \pi_{\ell+1}$ (where $V_{h + 1} := V_1$). For $h \geq 2$, a strong digraph $\Gamma$ is cyclically $h$-partite if and only if $h$ divides the index of imprimitivity (see, e.g., \cite[p. 70]{br1991}). 

A matrix $A \in \mat{n}{\bb{C}}$ is called \emph{h-cyclic} (terminology was introduced in \cite{r1936}) if $\dg{A}$ is cyclically $h$-partite and if $\dg{A}$ is cyclically $h$-partite with ordered partition $\Pi$, then $A$ is said to be \emph{h-cyclic with partition} $\Pi$ or that $\Pi$ \emph{describes the $h$-cyclic structure of A}. The ordered partition $\Pi = (\pi_1,\dots, \pi_h)$ is \emph{consecutive} if $\pi_1 = \{1,\dots, i_1\}$, $\pi_2 = \{i_1 + 1,\dots, i_2\},\dots, \pi_h = \{ i_{h - 1} + 1, \dots, n \}$. If $A$ is $h$-cyclic with consecutive ordered partition $\Pi$, then $A$ has the block form
\begin{align}
\begin{bmatrix} 
0 	& A_{12} 	& 0 		& \cdots 	& 0 		\\
0 	& 0 		& A_{23} 	& \ddots 	& \vdots	\\
\vdots & \vdots 	& \ddots 	& \ddots 	& 0		\\
0 	& 0		& \cdots	& 0 		& A_{(h-1)h}	\\
A_{h1} & 0 & 0 &\cdots & 0
\end{bmatrix}															\label{cyclic_form} % 3.1
\end{align}
where $A_{i,i+1} = A[\pi_i|\pi_{i+1}]$ (\cite[p. 71]{br1991}). For any $h$-cyclic matrix $A$, there exists a permutation matrix $P$ such that $P^\top AP$ is $h$-cyclic with consecutive ordered partition. The \emph{cyclic index} or \emph{index of cyclicity} of $A$ is the largest $h$ for which $A$ is $h$-cyclic.

An irreducible nonnegative matrix $A$ is \emph{primitive} if $\dg{A}$ is primitive, and the \emph{index of imprimitivity} of $A$ is the index of imprimitivity of $\dg{A}$.  If $A$ is irreducible and imprimitive with index of imprimitivity $h \geq 2$, then $h$ is the cyclic index of $A$, $\dg{A}$ is cyclically $h$-partite with ordered partition $\Pi = (\pi_1,\dots, \pi_h)$, and the sets $\pi_i$ are uniquely determined (up to cyclic permutation of the $\pi_i$) (see, for example, \cite[p. 70]{br1991}). Furthermore, $\dg{A^h}$ is the disjoint union of $h$ primitive digraphs on the sets of vertices $\pi_i$, $i = 1,\dots, h$ (see, e.g., \cite[\S 3.4]{br1991}).

Following \cite{h2009}, given an ordered partition $\Pi = \left( \pi_1,\dots,\pi_h \right)$ of $\{1, \dots, n\}$ into $h$ nonnempty subsets, the \emph{cyclic characteristic matrix}, denoted by  $\chi_\Pi$, is the $n \times n$ matrix whose $(i,j)$-entry is 1 if there exists $\ell \in \{1,\dots, h\}$ such that $i \in \pi_\ell$ and $j \in \pi_{\ell + 1}$, and 0 otherwise. For an ordered partition $\Pi = \left( \pi_1,\dots,\pi_h \right)$ of $\{1, \dots, n\}$ into $h$ nonnempty subsets, note that 
\begin{enumerate}[label=(\arabic*)]
\item $\chi_\Pi$ is $h$-cyclic and $\dg{\chi_\Pi}$ contains every arc $(i,j)$ for $i \in \pi_\ell$ and $j \in \pi_{\ell+1}$; and
\item $A \in \mat{n}{\bb{C}}$ is $h$-cyclic with ordered partition $\Pi$ if and only if $\dg{A} \subseteq \dg{\chi_\Pi}$.
\end{enumerate}

Finally, we recall the Perron-Frobenius Theorem for irreducible imprimitive matrices.

% Thm 3.1---------------------------------------------------------------------------------------------------------------------------------------------------------------------------
\begin{thm}[see, e.g., \cite{bp1994, hj1990}] \label{pftirr}
Let $A \in \mat{n}{\bb{R}}$, $n \geq 2$, and suppose that $A$ is irreducible, nonnegative, and $h$ is the cyclic index of $A$. Then
\begin{enumerate}[label=(\alph*)]
\item $\rho >0$;
\item $\rho \in \sig{A}$;
\item there exists a positive vector $x$ such that $Ax = \rho x$; 
\item $\rho$ is an algebraically (and hence geometrically) simple eigenvalue of $A$; and
\item $\peri{A} = \left\{ \rho \omega^k : k \in R(h) \right\}$.
\item $\omega^k \sig{A}=\sig{A}$ for $k \in R(h)$.
\end{enumerate}
\end{thm}

%-------------------------------------------------------------------------------------------------------------------------------------------------------------------------------------
\section{Jordan Chains of \emph{h}-cyclic matrices}
%-------------------------------------------------------------------------------------------------------------------------------------------------------------------------------------

Unless otherwise noted, in this section it is assumed that $A \in \mat{n}{\bb{C}}$ is nonsingular, $h$-cyclic with ordered-partition $\Pi$, and has the form \eqref{cyclic_form}.

The following lemma describes the Jordan structure of $A$.

% Lem 4.1---------------------------------------------------------------------------------------------------------------------------------------------------------------------------
\begin{lem} \label{jchainlemma}
For $i$, $j \in \bb{Z}$, let $\alpha_{ij} := (i-j)\bmod{h}$. 
\begin{enumerate}
% Item 1
\item If $\left\{ x_{\langle 0, j \rangle} \right\}_{j=1}^r$ is a right Jordan chain corresponding to $\lambda \in \sig{A}$, where $x_{\left< 0, j \right>}$ is partitioned conformably with $A$ as 
\begin{align*}
x_{\left< 0, j \right>} = 
\begin{bmatrix} 
x_{1j} \\ 
x_{2j} \\ 
\vdots \\ 
x_{hj} \end{bmatrix},~j=1,\dots,r,
\end{align*}
then, for $k \in R(h)$, the set 
\begin{align*}
\left\{
x_{\left< k, j \right>} :=
\begin{bmatrix} 
(\omega^k)^{\alpha_{1j}} x_{1j} \\ 
(\omega^k)^{\alpha_{2j}} x_{2j} \\ 
\vdots \\ 
(\omega^k)^{\alpha_{hj}} x_{hj} 
\end{bmatrix} 
\right\}_{j=1}^r
\end{align*}
is a right Jordan chain corresponding to $\lambda \omega^k$.
% Item 2
\item If $\left\{ y_{\left< j, 0 \right>} \right\}_{j = 1}^r$ is a left Jordan chain corresponding to $\lambda \in \sig{A}$, where $y_{\left< j , 0 \right>}$ is partitioned conformably with $A$ as 
\begin{align*}
y_{\left< j , 0 \right>}^\top = \begin{bmatrix} y_{j1}^\top & y_{j2}^\top & \cdots & y_{jh}^\top \end{bmatrix},~j = 1,\dots,r,
\end{align*}
then, for $k \in R(h)$, the set
\begin{align*}
\left\{  
y_{\left< j , k \right>}^\top :=
\begin{bmatrix} 
(\omega^k)^{\alpha_{j1}} y_{j1}^\top & 
(\omega^k)^{\alpha_{j2}} y_{j2}^\top 
& \cdots & (\omega^k)^{\alpha_{jh}} y_{jh}^\top 
\end{bmatrix} \right\}_{j = 1}^r 										
\end{align*}
is a left Jordan chain corresponding to $\lambda \omega^k$.
\end{enumerate}
\end{lem}

\begin{proof} 
For $k=0$, the result holds by hypothesis; thus we assume that $k >0$. 

The following facts are easily established:
\begin{align}
\alpha_{ij} &= \alpha_{(i+1)(j+1)},~\forall i,j \in \bb{Z}									\label{alpha1} 	\\	% 4.1
\alpha_{(i+1)j} &= \alpha_{i(j-1)} = (\alpha_{ij} + 1) \bmod{h},~\forall i,j \in \bb{Z}					\label{alpha2}	\\	% 4.2
\alpha_{ij} &= (\alpha_{i \ell} + \alpha_{\ell j})\bmod{h},~\forall i,j,\ell \in \bb{Z}					\label{alpha3}	\\	% 4.3
\alpha &\equiv \beta \bmod{h} \Longrightarrow (\omega^k)^\alpha = (\omega^k)^\beta,~k \in \bb{Z}.		\label{omega} 		% 4.4
\end{align}

For ease of notation, let $h+1 := 1$. As a consequence of the $h$-cyclic structure of $A$, for $i = 1, \dots, h$ and $j=1,\dots,r$, note that the $i\textsuperscript{th}$ component of the column vector $A x_{\left< 0, j \right>}$ is given by
\begin{align}
\left[ A x_{\left< 0, j \right>} \right]_i = A_{i (i + 1)} x_{(i + 1) j}.								\label{icompvector} % 4.5
\end{align} 
By hypothesis,  
\begin{align}
Ax_{\left< 0, 1 \right>} = \lambda x_{\left< 0, 1 \right>},										\label{jcinit}		% 4.6
\end{align}
so that, following \eqref{alpha2} -- \eqref{jcinit},  for $k=1,\dots,h-1$
\begin{align*}
\left[ Ax_{\left< k, 1 \right>} \right]_i 
&=  (\omega^k)^{\alpha_{(i + 1)1}} A_{i (i + 1)} x_{(i + 1) 1}	\\
&= \lambda (\omega^k)^{\alpha_{(i + 1)1}} x_{i1} 			\\
&= \lambda \omega^k (\omega^k)^{\alpha_{i1}} x_{i1} 		
= \left[ \lambda \omega^k x_{\left< k, 1 \right>} \right]_i,
\end{align*}
i.e., $\left( \lambda \omega^k , x_{\left< k, 1 \right>} \right)$ constitutes a right-eigenpair for $A$.

By hypothesis, for $j=2,\dots,r$, 
\begin{align}
Ax_{\left< 0, j \right>} = x_{\left< 0, j-1 \right>} + \lambda x_{\left< 0, j \right>},						\label{jchain} 	% 4.7
\end{align}
so that, following \eqref{icompvector} and \eqref{jchain}, 
\begin{align*}
\left[ A x_{\left< 0, j \right>} \right]_i = A_{i (i + 1)} x_{(i + 1) j} = x_{i (j - 1)} + \lambda x_{ij}.
\end{align*}
Following \eqref{alpha1} -- \eqref{omega}, \eqref{icompvector}, and \eqref{jchain}, for $i=1,\dots, h$ and $j=2,\dots,r$,
\begin{align*}
\left[ A x_{\left< k, j \right>} \right]_i 
&= (\omega^k)^{\alpha_{(i + 1) j}} A_{i(i+1)} x_{(i + 1) j} 								\\
&= (\omega^k)^{\alpha_{(i + 1) j}} \left( x_{i(j - 1)} + \lambda x_{ij} \right) 						\\
&= (\omega^k)^{\alpha_{i(j - 1)}} x_{i(j - 1)} + \lambda \omega^k \left( \omega^k \right)^{\alpha_{ij}} x_{ij},	
\end{align*}
i.e, 
\begin{align*}
Ax_{\left< k, j \right>} = x_{\left< k, j - 1 \right>} + \lambda \omega^k x_{\left< k , j \right>}.
\end{align*}

We now prove the second assertion; for ease of notation, let $1-1 := h$. As a consequence of the $h$-cyclic structure of $A$, for $j = 1,\dots,r$ and $i = 1, \dots, h$, the $i\textsuperscript{th}$ component of the row vector $y_{\left< j, 0\right>}^\top A$ is given by
\begin{align}
\left[ y_{\left< j , 0 \right>}^\top A \right]_i =  y_{j (i - 1)}^\top A_{(i - 1) i}.							\label{ivector2}	% 4.8
\end{align}
By hypothesis, 
\begin{align}
y_{\left< r, 0\right>}^\top A =  \lambda y_{\left< r, 0\right>}^\top,								\label{jcinit2} 	% 4.9
\end{align}
so that, following \eqref{alpha1} -- \eqref{omega}, \eqref{ivector2}, and \eqref{jcinit2}, for $k = 1, \dots, h - 1$,
\begin{align*}
\left[ y_{\left< r, k \right>}^\top A \right]_i 
&= (\omega^k)^{\alpha_{r(i - 1)}} y_{r (i - 1)}^\top A_{(i - 1) i}	\\
&= \lambda (\omega^k)^{\alpha_{r(i - 1)}} y_{ri}^\top 		\\
&= \lambda \omega^k (\omega^k)^{\alpha_{ri}} y_{ri}^\top	
= \left[ \lambda \omega^k y_{\left< r, k \right>}^\top \right]_i,
\end{align*}
i.e., $\left( \lambda \omega^k, y_{\left< r, k \right>}^\top \right)$ is a left-eigenpair for $A$.

By hypothesis, for $j = 1, \dots, r - 1$,
\begin{align}
y_{\left< j , 0 \right>}^\top A = \lambda y_{\left< j, 0 \right>}^\top + y_{\left< j + 1, 0 \right>}^\top, 			\label{jchain2} 	% 4.10
\end{align}
hence, following \eqref{ivector2} and \eqref{jchain2},
\begin{align*}
\left[ y_{\left< j , 0 \right>}^\top A \right]_i =  y_{j (i - 1)}^\top A_{(i - 1) i} = \lambda y_{ji}^\top + y_{(j + 1) i}^\top.
\end{align*}
Following \eqref{alpha1} -- \eqref{omega}, \eqref{ivector2}, and \eqref{jcinit2}, for $j = 1,\dots, r - 1$ and $i = 1,\dots,h$,
\begin{align*}
\left[ y_{\left< j, k \right>}^\top A \right]_i 
&= (\omega^k)^{\alpha_{j (i - 1)}} y_{j (i - 1)}^\top A_{(i - 1) i} 								\\
&= (\omega^k)^{\alpha_{j (i - 1)}} \left( \lambda y_{ji}^\top + y_{(j + 1) i}^\top \right)					\\
&= \lambda \omega^k (\omega^k)^{\alpha_{ji}} y_{ji}^\top + (\omega^k)^{\alpha_{(j + 1) i}} y_{(j + 1) i}^\top,	
\end{align*}
i.e, 
\begin{align*}
y_{\left< j , k \right>}^\top A = \lambda \omega^k y_{\left< j, k \right>}^\top + y_{\left< j + 1, k \right>}^\top,
\end{align*}
and the proof is complete.
\end{proof}

% Cor 4.2---------------------------------------------------------------------------------------------------------------------------------------------------------------------------
\begin{cor} \label{jblockcor}
If $\jordan{r}{\lambda}$ is a Jordan block of $J$, then $\jordan{r}{\lambda \omega^k}$ is a Jordan block of $J$ for $k \in R(h)$. 
\end{cor}
 
% Lem 4.3---------------------------------------------------------------------------------------------------------------------------------------------------------------------------
\begin{lem} \label{wmatriceslemma}
 For $k =0,1,\dots,h-1$ and $\ell=1,\dots,r$, the matrix
\begin{align*}
W_{k\ell}^1 
&:=
\omega^k
\begin{bmatrix}
(\omega^k)^{\alpha_{1\ell}}	\\
\vdots 				\\
(\omega^k)^{\alpha_{h\ell}}
\end{bmatrix}
\hspace{-3pt}
\begin{bmatrix}
(\omega^k)^{\alpha_{\ell1}} & \cdots & (\omega^k)^{\alpha_{\ell h}}
\end{bmatrix} 												\\									 
&= 
\circulant{(\overbrace{\omega^k, 1, (\omega^k)^{h-1}, \dots, (\omega^k)^2}^h)}; 
\end{align*}
and for $k=0,1,\dots,h-1$ and $\ell = 1,\dots,r-1$, the matrix
\begin{align*}
W_{k\ell}^2 
&:=
\begin{bmatrix}
(\omega^k)^{\alpha_{1\ell}}	\\
\vdots 				\\
(\omega^k)^{\alpha_{h\ell}}
\end{bmatrix}
\hspace{-3pt}
\begin{bmatrix}
(\omega^k)^{\alpha_{(\ell + 1)1}} & \cdots & (\omega^k)^{\alpha_{(\ell + 1) h}}
\end{bmatrix}												\\
&=
\circulant{(\overbrace{\omega^k, 1, (\omega^k)^{h-1}, \dots, (\omega^k)^2}^h)}. 
\end{align*}
\end{lem}

\begin{proof} The result is trivial when $k=0$, so we assume that $k > 0$.

To establish the first claim, for $\ell = 1,\dots, r$ and $k=1,\dots,h-1$
\begin{align*}
W_{k\ell}^1
=
\omega^k
\begin{bmatrix}
(\omega^k)^{\alpha_{i \ell} + \alpha_{\ell j}}
\end{bmatrix}_h^{i,j=1},
\end{align*}
however, following \eqref{alpha2} -- \eqref{omega}, note that 
\begin{align*}
(\omega^k)^{\alpha_{i \ell} + \alpha_{\ell j}} = (\omega^k)^{(\alpha_{i \ell} + \alpha_{\ell j}) \bmod{h}}	 = (\omega^k)^{\alpha_{ij}},
\end{align*}
whence it follows that
\begin{align*}
W_{k\ell}^1
=
\omega^k
\begin{bmatrix}
(\omega^k)^{\alpha_{ij}}
\end{bmatrix}_h^{i,j=1}
=
\begin{bmatrix}
(\omega^k)^{\alpha_{ij} + 1}
\end{bmatrix}_h^{i,j=1}.
\end{align*}
Following \eqref{alpha2} and \eqref{omega}, $(\omega^k)^{\alpha_{ij} + 1} = (\omega^k)^{{(\alpha_{ij} + 1)}\bmod{h}} = (\omega^k)^{\alpha_{(i + 1)j}}$. Thus
\begin{align*}
W_{k\ell}^1 
&= 
\begin{bmatrix}
(\omega^k)^{\alpha_{(i+1)j}}
\end{bmatrix}_h^{i,j=1}													\\
&=															
\kbordermatrix{
 & 1 & 2 & 3 & \cdots & h-1 & h \\
1 & \omega^k & 1 & (\omega^k)^{h-1} & \cdots & (\omega^k)^3 & (\omega^k)^2			\\
2 & (\omega^k)^2 & \omega^k & 1 & \cdots & (\omega^k)^4 & (\omega^k)^3				\\
3 & (\omega^k)^3 & (\omega^k)^2 & \omega^k & \ddots & \vdots & \vdots 				\\
\vdots & \vdots & \vdots & \ddots & \ddots & \ddots & \vdots						\\
h-1 & (\omega^k)^{h-1} & (\omega^k)^{h-2} & (\omega^k)^{h-3} & \cdots & \omega^k & 1	\\
h & 1 & (\omega^k)^{h-1} & (\omega^k)^{h-2} & \cdots & (\omega^k)^2 & \omega^k
}																\\
&= \circulant{(\omega^k, 1, (\omega^k)^{h-1}, \dots, (\omega^k)^2)},
\end{align*} 
and the first claim is established.

To establish the second claim, for $\ell=1,\dots,r-1$ and $k = 1, \dots, h-1$ note that 
\begin{align*}
W_{k\ell}^2
= 
\begin{bmatrix}
(\omega^k)^{\alpha_{i \ell} + \alpha_{(\ell + 1) j}}
\end{bmatrix}_h^{i,j=1},
\end{align*}
however, following \eqref{alpha2} -- \eqref{omega},
\begin{align*}
(\omega^k)^{\alpha_{i \ell} + \alpha_{(\ell + 1) j}} 
= (\omega^k)^{\left( \alpha_{i \ell} + \alpha_{(\ell + 1) j} \right) \bmod{h}}			
= (\omega^k)^{\alpha_{(i+1)j}},
\end{align*}
whence it follows that
\begin{align*}
W_{k\ell}^2 = \circulant{(\omega^k, 1, (\omega^k)^{h-1}, \dots, (\omega^k)^2)}
\end{align*}
and the proof is complete.
\end{proof}

% Rem 4.4--------------------------------------------------------------------------------------------------------------------------------------------------------------------------
\begin{rem} \label{cremark}
For ease of notation, let 
\[ C_k := \circulant{(\omega^k, 1, (\omega^k)^{h-1}, \dots, (\omega^k)^2)},~k\in R(h). \] 
Because $\omega^\ell$ is an $h\textsuperscript{th}$-root of unity for $\ell \in \bb{Z}$, note that, for $\ell = 1, \dots, h-1$,  
\begin{align*}
\sum_{k=0}^{h-1} (\omega^k)^\ell = \sum_{k=0}^{h-1} (\omega^\ell)^k = \frac{(\omega^\ell)^h - 1}{\omega^\ell - 1} = 0.
\end{align*}
Thus, 
\begin{align*}
\sum_{k = 0}^{h-1}
C_k														
=\kbordermatrix{
  & 1 & 2 & 3 & \cdots & h \\
1 & & h & 			\\
2 & & & h &			\\
\vdots & & & & \ddots 	\\
h-1 & & & & & h 		\\
h & h} 							
= h K_h.							 			
\end{align*}
\end{rem}

% Rem 4.5--------------------------------------------------------------------------------------------------------------------------------------------------------------------------
\begin{rem} 
With $\nu_h$ as defined in \eqref{omegastar} and
\begin{align*}
J \left( \lambda \nu_h, r \right) := 
\begin{bmatrix}
\jordan{r}{\lambda} 				\\
& \jordan{r}{\lambda \omega} 			\\
& & \ddots 						\\
& & & \jordan{r}{\lambda \omega^{h-1}}
\end{bmatrix} \in \mat{rh}{\bb{C}},
\end{align*}
it follows that a Jordan form of a nonsingular $h$-cylic matrix $A$ has the form
\begin{align*}
\inv{Z} A Z = J = \bigoplus_{i=1}^{t'} J\left( \lambda_i \nu_h, r_i \right),~t'|t
\end{align*}
and the Jordan chains comprising the matrix $Z$ may be selected as in \hyp{jchainlemma}{Lemma}; i.e., if $\left\{ x_{\left<0, j \right>} \right\}_{j=1}^r$ is a left Jordan chain for $\lambda$ and $\left\{ y_{\left< j , 0 \right>} \right\}_{j=1}^r$ is a right Jordan chain for $\lambda$, then $\left\{ x_{\left< k, j \right>} \right\}_{j=1}^r$ is a left Jordan chain for $\lambda \omega^k$ and $\left\{ y_{\left< j , k \right>} \right\}_{j = 1}^r$ is a right Jordan chain for $\lambda \omega^k$, $k = 1, \dots, h-1$.
\end{rem}

We are now ready to present the following. 

% Thm 4.6---------------------------------------------------------------------------------------------------------------------------------------------------------------------------
\begin{thm}
For $i=1,\dots,t'$, if 
\begin{align}
A_{\lambda_i} :=
Z
\diag{0, \dots, 0, \overbrace{J\left( \lambda_i \nu_h, r_i \right)}^i,0,\dots,0}
\inv{Z} \in \mat{n}{\bb{C}},												\label{cycliccommutatormatrix} % 4.11
\end{align}
then $\dg{A_{\lambda_i}} \subseteq \dg{\chi_\Pi}$, $A_{\lambda_i}$ commutes with $A$, and $A_{\lambda_i} A_{\lambda_j} = A_{\lambda_j} A_{\lambda_i} = 0$ for $i \neq j$, $j=1,\dots,t'$.
\end{thm}

\begin{proof} 
Let $W_{kj}^1$ and $W_{kj}^2$ be defined as in \hyp{wmatriceslemma}{Lemma}, and let $C_k$ be defined as in \hyp{cremark}{Remark}. By definition,
\begin{align*}
A_{\lambda_i}  
&= % 1-------------------------------------------------------------------------------------------------------------------------------------------------------------------------------
\sum_{k=0}^{h-1} \left( \sum_{j=1}^{r_i}  \lambda_i \omega^k x_{\left<k, j \right>} y_{\left< j , k \right>}^\top + \sum_{j=1}^{r_i - 1}  x_{\left<k, j \right>} y_{\left< j+1 , k \right>}^\top\right) 						\\ 
&= % 2-------------------------------------------------------------------------------------------------------------------------------------------------------------------------------
\sum_{k=0}^{h-1} 
\left(
\sum_{j=1}^{r_i}
\lambda_i \omega^k 	
\begin{bmatrix} 
(\omega^k)^{\alpha_{1j}} x_{1j} \\ 
\vdots \\ 
(\omega^k)^{\alpha_{hj}} x_{hj} 
\end{bmatrix} 
\hspace{-3pt}
\begin{bmatrix} 
(\omega^k)^{\alpha_{j1}} y_{j1}^\top & \cdots & (\omega^k)^{\alpha_{jh}} y_{jh}^\top 
\end{bmatrix} \right.
+																	\\
& \qquad \left.
\sum_{j=1}^{r_i-1}
\begin{bmatrix} 
(\omega^k)^{\alpha_{1j}} x_{1j} \\ 
\vdots \\ 
(\omega^k)^{\alpha_{hj}} x_{hj} 
\end{bmatrix} 
\hspace{-3pt}
\begin{bmatrix} 
(\omega^k)^{\alpha_{(j+1)1}} y_{(j+1)1}^\top & \cdots & (\omega^k)^{\alpha_{(j+1)h}} y_{(j+1)h}^\top 
\end{bmatrix}  
\right)																	\\
&= % 3------------------------------------------------------------------------------------------------------------------------------------------------------------------------------- 
\lambda_i
\sum_{k=0}^{h-1} 
\sum_{j=1}^{r_i} 
W_{kj}^1  
\circ 
x_{\langle 0, j \rangle} y_{\langle j, 0 \rangle}^\top 
+
\sum_{k=0}^{h-1}
\sum_{j=1}^{r_i-1}
W_{kj}^2 
\circ 
x_{\langle 0, j \rangle} y_{\langle j+1, 0  \rangle}^\top.	
\end{align*}
Following \hyp{wmatriceslemma}{Lemma}, and upon reversing the sums, 
\begin{align*}
A_{\lambda_i}
&= % 1-------------------------------------------------------------------------------------------------------------------------------------------------------------------------------
\lambda_i \sum_{j=1}^{r_i} \left( \sum_{k=0}^{h-1} C_k \circ x_{\langle 0, j \rangle} y_{\langle j, 0 \rangle}^\top \right) +
\sum_{j=1}^{r_i - 1} \left( \sum_{k=0}^{h-1} C_k \circ x_{\langle 0, j \rangle} y_{\langle j+1, 0 \rangle}^\top \right) 				\\
&= % 2-------------------------------------------------------------------------------------------------------------------------------------------------------------------------------
\lambda_i \sum_{j=1}^{r_i} \left[ \left( \sum_{k=0}^{h-1} C_k \right) \circ x_{\langle 0, j \rangle} y_{\langle j, 0 \rangle}^\top \right] +				
\sum_{j=1}^{r_i - 1} \left[ \left( \sum_{k=0}^{h-1} C_k \right) \circ x_{\langle 0, j \rangle} y_{\langle j+1, 0 \rangle}^\top \right] 		\\
&= % 3-------------------------------------------------------------------------------------------------------------------------------------------------------------------------------
\lambda_i h \sum_{j=1}^{r_i} K_h \circ x_{\langle 0, j \rangle} y_{\langle j, 0 \rangle}^\top + 
h \sum_{j=1}^{r_i-1} K_h \circ x_{\langle 0, j \rangle} y_{\langle j+1, 0 \rangle}^\top 								\\
&= % 4-------------------------------------------------------------------------------------------------------------------------------------------------------------------------------
\lambda_i h \sum_{j=1}^{r_i}
\begin{bmatrix}
0 & x_{1j} y_{j2}^\top & \cdots & \cdots & 0 	\\
0 & 0 & x_{2j} y_{j3}^\top & \cdots & 0 	\\
\vdots & \vdots & \ddots & \ddots & \vdots	\\
0 & 0 & \cdots & 0 & x_{(h-1)j} y_{jh}^\top	\\
x_{hj} y_{j1}^\top & 0 & \cdots & 0 & 0
\end{bmatrix} +															\\
& \qquad
h \sum_{j=1}^{r_i - 1}
\begin{bmatrix}
0 & x_{1j} y_{(j+1)2}^\top & \cdots & \cdots & 0 	\\
0 & 0 & x_{2j} y_{(j+1)3}^\top & \cdots & 0 		\\
\vdots & \vdots & \ddots & \ddots &  \vdots		\\
0 & 0 & \cdots &  0 & x_{(h-1)j} y_{(j+1)h}^\top	\\
x_{hj} y_{(j+1)1}^\top & 0 &  \cdots & 0 & 0
\end{bmatrix} 															\\
& = % 5------------------------------------------------------------------------------------------------------------------------------------------------------------------------------
\begin{bmatrix}
0 & A_{12}^{(i)} & \cdots & \cdots & 0 		\\
0 & 0 & A_{23}^{(i)} & \cdots & 0 		\\
\vdots & \vdots & \ddots & \ddots  & \vdots	\\
0 & 0 & \cdots &  0 & A_{(h-1)h}^{(i)}		\\
A_{h1}^{(i)} & 0 &  \cdots & 0 & 0
\end{bmatrix},							
\end{align*}
where 
\begin{align}
A_{\ell(\ell+1)}^{(i)} := \lambda_i h \sum_{j=1}^{r_i} x_{\ell j} y_{j (\ell + 1)}^\top + h \sum_{j=1}^{r_i - 1} x_{\ell j} y_{(j+1)(\ell + 1)}^\top, \label{cycmatcomp}
\end{align}
for $i=1,\dots,t'$ and $\ell=1,\dots,h$ (where $h+1:=1$). 

The conclusion that $\dg{A_{\lambda_i}} \subseteq \dg{\chi_\Pi}$ follows because the vectors in the sets $\left\{ x_{\left< 0, j \right>} \right\}_{j=1}^r$ and $\left\{ y_{\left< j, 0 \right>} \right\}_{j=1}^r$ are partitioned conformably with $\Pi$; the matrices $A_{\lambda_i}$ and $A$ commute because $A_{\lambda_i}$ and $A$ are simultaneously triangularizable; and $A_{\lambda_i} A_{\lambda_j} = A_{\lambda_j} A_{\lambda_i} =0$ by construction of the matrices $A_{\lambda_i}$ and $A_{\lambda_j}$. 
\end{proof}

% Cor 4.7---------------------------------------------------------------------------------------------------------------------------------------------------------------------------
\begin{cor} \label{cycliccor}
If $x$ is a strictly nonzero right eigenvector and $y$ is a strictly nonzero left eigenvector corresponding to $\lambda \in \bb{C}$, then $A_\lambda$ has cyclic index $h$ and $\dg{A_\lambda} = \dg{\chi_\Pi}$.
\end{cor}

% Ex 4.8---------------------------------------------------------------------------------------------------------------------------------------------------------------------------
\begin{ex}
As a special case, we examine the following: let $A$ be a nonnegative, irreducible, imprimitive, nonsingular matrix with index of cyclicity $h$ and assume, without loss of generality, that $\sr{A} = 1$. Following \hyp{pftirr}{Theorem} and \hyp{jblockcor}{Corollary}, note that a Jordan form of $A$ is
\begin{align*}
\inv{Z} A Z = 
\begin{bmatrix}
J(\nu_h,1)	& 							\\
			& J\left( \lambda_2 \nu_h,r_2 \right) & &	\\
			& & \ddots 						\\
			& & & J\left( \lambda_t \nu_h, r_t \right)
\end{bmatrix}.													
\end{align*}
Consider the matrix 
\begin{align*}
A_1 := Z \begin{bmatrix} J(\nu_h,1) & 0 \\ 0 & 0 \end{bmatrix} \inv{Z}. 
\end{align*}
Following \hyp{cycliccor}{Corollary}, $\dg{A_1} = \dg{\chi_\Pi}$ and $AA_1 = A_1 A$; moreover, following \hyp{pftirr}{Theorem}, there exist positive vectors $x$ and $y$ such that $Ax = x$ and $y^\top A = y^\top$. If $x$ and $y$ are partitioned conformably with $A$ as
\begin{align*}
x =
\begin{bmatrix}
x_1 	\\
x_2 	\\
\vdots \\
x_h
\end{bmatrix} \text{and }
y^\top = 
\begin{bmatrix}
y_1^\top & y_2^\top & \cdots & y_h^\top 
\end{bmatrix}, 
\end{align*}
then, following \eqref{cycmatcomp},  
\begin{align*}
A_1 = h
\begin{bmatrix}
0 & x_1 y_2^\top & \cdots & \cdots & 0 		\\
0 & 0 & x_2 y_3^\top & \cdots & 0 			\\
\vdots & \vdots & \ddots & \ddots  & \vdots	\\
0 & 0 & \cdots &  0 & x_{h-1} y_h^\top		\\
x_h y_1^\top & 0 &  \cdots & 0 & 0
\end{bmatrix} \geq 0.
\end{align*}
\end{ex}

% Ex 4.9---------------------------------------------------------------------------------------------------------------------------------------------------------------------------
\begin{ex} 
If
\begin{align*}
A :=
\frac{1}{3}
\begin{bmatrix}
0 & 0 & 1 & 2 & 0 & 0 	\\
0 & 0 & 2 & 1 & 0 & 0	\\
0 & 0 & 0 & 0 & 1 & 2	\\
0 & 0 & 0 & 0 & 2 & 1	\\
1 & 2 & 0 & 0 & 0 & 0	\\
2 & 1 & 0 & 0 & 0 & 0
\end{bmatrix} \in \mat{6}{\bb{R}},
\end{align*}
then one can verify that $A=ZD\inv{Z}$, where 
\begin{align*}
Z =
\left[
\begin{array}{*{6}{r}}
1 & 1 & 1 & 1 & 1 & 1						\\
1 & 1 & 1 & -1 & -1 & -1						\\
1 & \omega & \omega^2 & 1 & \omega & \omega^2		\\
1 & \omega & \omega^2 & -1 & -\omega & -\omega^2	\\
1 & \omega^2 & \omega & 1 & \omega^2 & \omega		\\
1 & \omega^2 & \omega & -1 & -\omega^2 & -\omega
\end{array}
\right]
\end{align*}
and
\begin{align*}
D=\diag{1,\omega,\omega^2,-\frac{1}{3},-\frac{\omega}{3},-\frac{\omega^2}{3}}.
\end{align*}

Following \eqref{cycliccommutatormatrix},
\begin{align*}
A_1 = Z\diag{1,\omega,\omega^2,0_{3\times3}} \inv{Z} =
\frac{1}{2}
\begin{bmatrix}
0 & 0 & 1 & 1 & 0 & 0 	\\
0 & 0 & 1 & 1 & 0 & 0 	\\
0 & 0 & 0 & 0 & 1 & 1 	\\
0 & 0 & 0 & 0 & 1 & 1 	\\
1 & 1 & 0 & 0 & 0 & 0 	\\
1 & 1 & 0 & 0 & 0 & 0
\end{bmatrix}
\end{align*}
and
\begin{align*}
A_{-\frac{1}{3}} = -\frac{1}{3}Z\diag{0_{3\times3},1,{\omega},\omega^2} \inv{Z} =
\frac{1}{6}
\left[
\begin{array}{*{6}{r}}
0 & 0 & -1 & 1 & 0 & 0  	\\
0 & 0 & 1 & -1 & 0 & 0 	\\
0 & 0 & 0 & 0 & -1 & 1 	\\
0 & 0 & 0 & 0 & 1 & -1 	\\
-1 & 1 & 0 & 0 & 0 & 0  	\\
1 & -1 & 0 & 0 & 0 & 0
\end{array}
\right].
\end{align*}

Following \hyp{cycliccor}{Corollary}, note that 
\[ 
A A_1 
= A_1 A
= \frac{1}{2}
\begin{bmatrix}
0 & 0 & 0 & 0 & 1 & 1 	\\
0 & 0 & 0 & 0 & 1 & 1 	\\
1 & 1 & 0 & 0 & 0 & 0 	\\
1 & 1 & 0 & 0 & 0 & 0 	\\
0 & 0 & 1 & 1 & 0 & 0 	\\
0 & 0 & 1 & 1 & 0 & 0 	
\end{bmatrix}
\]
and 
\[ 
A A_{-\frac{1}{3}} 
= A_{-\frac{1}{3}}A
=\frac{1}{18}
\left[
\begin{array}{*{6}{r}}
0 & 0 & 0 & 0 & 1 & -1 	\\
0 & 0 & 0 & 0 & -1 & 1 	\\
1 & -1 & 0 & 0 & 0 & 0 	\\
-1 & 1 & 0 & 0 & 0 & 0 	\\
0 & 0 & 1 & -1 & 0 & 0 	\\
0 & 0 & -1 & 1 & 0 & 0	
\end{array}
\right].
\]
\end{ex}

%------------------------------------------------------------------------------------------------------------------------------------------------------
% Bib
%------------------------------------------------------------------------------------------------------------------------------------------------------

\bibliographystyle{abbrv}
\bibliography{laabib,crs}

\end{document}